\documentclass{amsart}
\usepackage{amsmath, amssymb, amsthm}

\newtheorem{thm}{Theorem}[section]
\newtheorem{lem}[thm]{Lemma}
\theoremstyle{remark}
\newtheorem{rmk}[thm]{Remark}

\def\area{{\textrm{area}}}
\def\C{{\mathbb C}}
\def\id{{\textrm{id}}}
\def\p{{\partial}}
\def\pbar{{\overline{\partial}}}
\def\R{{\mathbb R}}
\def\Z{{\mathbb Z}}

\title[Holomorphic disks with pre-Lagrangian boundary]{$J$-holomorphic disks with pre-Lagrangian boundary conditions}

\author{Stefan M\"uller}

\address{Georgia Southern University, Department of Mathematical Sciences, 65 Georgia Ave.\ Room 3008, P.O.\ Box 8093, Statesboro, GA 30460, USA}

\email{smueller@georgiasouthern.edu}



\begin{document}
\thispagestyle{plain}

\begin{abstract}
The purpose of this paper is to carry out a classical construction of a non-constant holomorphic disk with boundary on (the suspension of) a Lagrangian submanifold in $\R^{2 n}$ in the case the Lagrangian is the lift of a coisotropic (a.k.a.\ pre-Lagrangian) submanifold in (a subset $U$ of) $\R^{2 n - 1}$.
We show that the positive lower and finite upper bounds for the area of such a disk (which are due to M.~Gromov and J.-C.~Sikorav and F.~Laudenbach-Sikorav for general Lagrangians) depend on the coisotropic submanifold only but not on its lift to the symplectization.
The main application is to a $C^0$-characterization of contact embeddings in terms of coisotropic embeddings in another paper by the present author.
Moreover, we prove a version of Gromov's non-existence of exact Lagrangian embeddings into standard $\R^{2 n}$ for coisotropic embeddings into $S^1 \times \R^{2 n}$.
This allows us to distinguish different contact structures on the latter by means of the (modified) contact shape invariant.
As in the general Lagrangian case, all of the existence results are based on Gromov's theory of $J$-holomorphic curves and his compactness theorem (or persistence principle).
Analytical difficulties arise mainly at the ends of the cone $\R_+ \times U$.
\end{abstract}

\maketitle

\section{Introduction and main results} \label{sec:intro}

Throughout this paper, let $(M, \xi)$ be a (connected) cooriented contact manifold without boundary of dimension $2 n - 1$, $\alpha$ be a contact form with $\ker \alpha = \xi$ that induces the given coorientation of the contact structure $\xi$, and $L$ be a closed (and connected) smooth $n$-dimensional manifold.
We consider embeddings $\iota \colon L \hookrightarrow M$ so that the one-form $f \, \iota^* \alpha$ on $L$ is closed for some positive function $f$ on $L$.
Such embeddings are called coisotropic or pre-Lagrangian.
See subsection~\ref{subsec:coisotropic} for details.
We can identify $M$ with the subset $1 \times M$ in its symplectization $\R_+ \times M$.
If $U$ is an open subset of $\R^{2 n - 1}$ with its standard contact structure, its symplectization can be identified with a cone $CU$ in $\R^{2 n}$ with its standard symplectic structure.

\begin{thm} \label{thm:non-constant-disk}
Let $U \subset \R^{2 n - 1}$ be an open subset, $\iota \colon L \hookrightarrow U$ be a coisotropic embedding, $H \colon [0, 1] \times U \to \R$ be a compactly supported contact Hamiltonian so that $(\varphi_H^1 \circ \iota) (L) \cap \iota (L) = \emptyset$, and $\epsilon > 0$.
Denote by $\Psi \colon L \times S^1 \to U \times \R^2$ the induced double suspension.
Then there exists a non-constant holomorphic disk $(D, \p D) \to (\R^{2 n + 1}, \Psi (L \times S^1))$ of area $a$ such that $0 < a < 4 \| H \| + \epsilon$, and the obvious map $\p D \to \Psi (L \times S^1) \to L \times S^1 \to L$ represents a nonzero element of $H_1 (L, \R)$.
\end{thm}

On the other hand, the area of such a disk is also bounded from below.

\begin{lem}[{\cite{laudenbach:hdl94}}] \label{topo-area-bound}
If $K \subset \R^d$ is a compact submanifold, then there exists a constant $b = b (K) > 0$ so that any disk with non-trivial boundary in $H_1 (K, \R)$ has area $a \ge b$.
\end{lem}

Our main interest is in (tubular neighborhoods of embedded) $n$-dimensional submanifolds of $\R^{2 n - 1}$ that admit a nowhere tangent contact vector field.
When $n = 2$, such a submanifold is called a convex surface.

\begin{thm} \label{thm:empty-shape}
Let $\jmath \colon L \hookrightarrow \R^{2 n - 1}$ be a (necessarily non-coisotropic) embedding and $X = X_H$ be a contact vector field that is nowhere tangent to the image $\jmath (L)$.
Then there exists a tubular neighborhood $N$ of $\jmath (L)$ with the property that there exists no coisotropic embedding $\iota \colon L \hookrightarrow N$ such that the induced homomorphism $\iota_* \colon H_1 (L, \R) \to H_1 (N, \R)$ is injective.
In particular, there exists no coisotropic embedding $\iota \colon L \hookrightarrow N$ with $\iota_* = \jmath_* \colon H_1 (L, \R) \to H_1 (N, \R)$.
\end{thm}

A coisotropic embedding $L \hookrightarrow M$ lifts to a Lagrangian embedding of $L$ into the symplectization $\R_+ \times M$, and if $U$ is an open subset of $\R^{2 n - 1}$, the latter can be identified with a cone in $\R^{2 n}$.
If the Lagrangian submanifold of (the cone in) $\R^{2 n}$ is displaceable, then a classical construction due to Gromov \cite{gromov:pcs85} and Sikorav \cite{sikorav:qpp91} gives rise to a non-constant holomorphic disk with boundary on (the suspension of) the Lagrangian submanifold, and an upper bound for the area of such a disk.
On the other hand, this area is bounded from below by a positive constant if either the Lagrangian submanifold is rational or if one imposes restrictions on its homotopy type.
Note however that if $\iota \colon L \hookrightarrow M$ is as above, then the one-form $s f \, \iota^* \alpha$ on $L$ is also closed for all $s > 0$, and the above upper and lower bounds depend on this parameter $s$.
The purpose of this paper is to carry out the Gromov and Sikorav construction for (Lagrangian lifts of) coisotropic embeddings into open subsets $U$ of $\R^{2 n - 1}$, and to refine the corresponding estimates to prove the aforementioned theorems.
We moreover prove the following version of Gromov's theorem regarding the non-existence of exact Lagrangian embeddings into standard $\R^{2 n}$.

\begin{thm} \label{thm:non-exact-coiso}
Let $\iota \colon S^1 \times L \hookrightarrow S^1 \times \R^{2 n}$ be a coisotropic embedding and $f$ be a positive function on $S^1 \times L$ so that the one-form $f \, \iota^* \alpha_{\textrm{can}}$ is closed, and suppose that $\iota^* = \iota_0^* \colon H^1 (S^1 \times \R^{2 n}, \R) \to H^1 (S^1 \times L, \R)$, where $\iota_0$ denotes the canonical map $S^1 \times L \to S^1 \times 0 \to S^1 \times \R^{2 n}$.
Consider the corresponding cohomology class $[f \, \iota^* \alpha_{\textrm{can}}] = (A, Z) \in H^1 (S^1, \R) \oplus H^1 (L, \R)$.
Then $Z \not= 0$.
\end{thm}

The main applications of these theorems are to a $C^0$-characterization of contact embeddings and to the calculation of the (modified) shape invariant of $S^1 \times U$ for open subsets $U$ of $\R^{2 n - 1}$ in \cite{mueller:csc16}.
The present paper is organized as follows.

Section~\ref{sec:contact} presents necessary elements of contact geometry and symplectization, and section~\ref{sec:suspension} explains the suspension construction on a coisotropic embedding.
Section~\ref{sec:holo-disk} reviews Gromov's theory of $J$-holomorphic curves, and contains the proofs of Theorems~\ref{thm:non-constant-disk} and \ref{thm:empty-shape}.
Rational coisotropic embeddings and energy-capacity type inequalities are discussed in Section~\ref{sec:limitations}.
The final section~\ref{sec:proof} generalizes some of the theory to embeddings into products $Y \times \R^2$, and contains the proof of Theorem~\ref{thm:non-exact-coiso}.

\section{Contact geometry and symplectization} \label{sec:contact}

\subsection{Symplectization} \label{subsec:symplectization}
The symplectization of $(M, \alpha)$ is the symplectic manifold $\R_+ \times M$ with symplectic structure $\omega = d (t \pi_2^* \alpha)$, where $t$ denotes the coordinate on the first factor, and $\pi_2 \colon \R_+ \times M \to M$ denotes the projection to the second factor.
Up to an exact symplectic diffeomorphism, the symplectization depends only on the contact structure $\xi$ and not on the particular choice of contact form $\alpha$.

Let $U$ be an open subset of $\R^{2 n - 1}$ with its standard contact structure $\xi_0 = \ker \alpha_0$, where
\[ \alpha_0 = dz + \frac{1}{2} \sum_{j = 1}^{n - 1} (x_j dy_j - y_j dx_j) = dz + \frac{1}{2} \sum_{j = 1}^{n - 1} r_j^2 d\theta_j, \]
and $x_j = r_j \cos \theta_j$ and $y_j = r_j \sin \theta_j$ denote (rectangular and polar) coordinates on $\R^{2 n - 2}$.
We can identify the symplectization $(\R_+ \times U, d (t \pi_2^* \alpha_0))$ with a cone in $\R_+ \times \R^{2 n - 1} \subset \R^{2 n}$ via the symplectic embedding
\begin{align} \label{eqn:rho}
\rho (t, z, r_1, \theta_1, \ldots, r_{n - 1}, \theta_{n - 1}) = (t, z, \sqrt{t} \, r_1, \theta_1, \ldots, \sqrt{t} \, r_{n - 1}, \theta_{n - 1}) .
\end{align}
Note that the latter pulls back the one-form
\begin{align} \label{eqn:lambda}
\lambda_0 = x_1 dy_1 + \frac{1}{2} \sum_{j = 2}^n (x_j dy_j - y_j dx_j) = x_1 dy_1 + \frac{1}{2} \sum_{j = 2}^n r_j^2 d\theta_j
\end{align}
to the one-form $\rho^* \lambda_0 = t \pi_2^* \alpha_0$, and thus is indeed symplectic with respect to the standard symplectic form $\omega_0 = d\lambda_0$ on $\R^{2 n}$.
We denote the image of this embedding by $CU = \rho (\R_+ \times U) \subset \R^{2 n}$.

The tangent bundle $T (\R_+ \times U)$ decomposes naturally into the direct sum bundle $T\R_+ \oplus T\R \oplus T\R^2 \oplus \cdots \oplus T\R^2$, and we call a tangent vector planar if it is tangent to one of the $T\R^2$ factors.
We also call the image of a planar vector under the homomorphism $d\rho$ planar.
An important role will later be played by the projection map $\Pi \colon CU \to CU$ defined by $\Pi (\rho (t, z, x, y)) = \rho (1, z, x, y)$.
For later reference, we observe that if $v \in T_{\rho (t, z, x, y)} CU$ is a planar vector, then $d\Pi (v) = (1 \, / \sqrt{t}) v$.

\subsection{Contact isotopies}
Denote by $R_\alpha$ the Reeb vector field of the contact form $\alpha$.
Let $H \colon [0, 1] \times M \to \R$ be a compactly supported smooth function.
The identities $d\alpha (X_H^s, \cdot) = (dH_s (R_\alpha)) \alpha - dH_s$ and $\alpha (X_H^s) = H_s$ define a unique (time-dependent) vector field $X_H^s$ on $M$.
This vector field integrates to an isotopy (or path) of diffeomorphisms $\{ \varphi_H^s \}$ that preserve the (cooriented) contact structure $\xi = \ker \alpha$.
That is, there exist positive functions $h_s$ on $M$ such that $(\varphi_H^s)^* \alpha = h_s \alpha$.

Let $(W, \omega)$ be a symplectic manifold without boundary, and $F \colon [0, 1] \times W \to \R$ be a compactly supported smooth function.
The identity $\omega (X_F^s, \cdot) = - dF_s$ defines a unique (time-dependent) vector field $X_F^s$ on $W$.
This vector field integrates to an isotopy (or path) of diffeomorphisms $\{ \varphi_F^s \}$ that preserve the symplectic form $\omega$.
A contact isotopy $\{ \varphi_H^s \}$ lifts to a Hamiltonian isotopy $\{ \varphi_F^s \}$ on the symplectization $\R_+ \times M$ defined by $\varphi_F^s (t, x) = (t / h_s (x), \varphi_H^s (x))$, and their Hamiltonian functions are related by the identity $F_s (t, x) = t H_s (x)$.
Note that although $F$ is not compactly supported (unless $H = 0$), the vector field $X_F$ is nevertheless uniquely integrable.

\subsection{Coisotropic embeddings} \label{subsec:coisotropic}
An embedding $\iota \colon L \hookrightarrow M$ is called coisotropic if $\iota$ is transversal to the contact structure $\xi$, and the distribution $\iota^* \xi = \ker (\iota^* \alpha) \subset TL$ can be defined by a closed one-form.
Any such one-form must be of the form $f \, \iota^* \alpha$ for a non-zero function $f$ on $L$.
By replacing $f$ by $- f$ if necessary, we may assume that $f$ is positive.
Note that the one-form $s f \, \iota^* \alpha$ is also closed for any $s \not= 0$.
However, since a smooth function on a closed manifold must have a critical point, the codimension one distribution $\iota^* \xi$ cannot be defined by an exact one-form.
Coisotropic embeddings are also called pre-Lagrangian embeddings in the literature.
As in \cite{mueller:csc16} (which in turn adapts this convention from Y.~Eliashberg), we continue to call them coisotropic in this paper.
The term pre-Lagrangian is used in the title of this paper solely to avoid possibly confusion with coisotropic submanifolds of symplectic manifolds (of lower codimension).

An embedding $\iota \colon L \hookrightarrow W$ is called Lagrangian if $\iota^* \omega = 0$ and $\dim W = 2 n$.
A coisotropic embedding $\iota \colon L \hookrightarrow M$ with function $f$ as above lifts to a Lagrangian embedding $\hat{\iota}$ of $L$ into the symplectization $\R_+ \times M$ given by $\hat{\iota} (x) = (f (x), \iota (x))$, since $\hat{\iota}^* (t \pi_2^* \alpha) = f \, \iota^* \alpha$.
In particular, $[\hat{\iota}^* (t \pi_2^* \alpha)] = [f \, \iota^* \alpha] \in H^1 (L, \R)$.

\subsection{Products with exact symplectic manifolds}
Let $(M, \xi = \ker \alpha)$ be a cooriented contact manifold of dimension $2 n - 1$, and $(W, \omega = d\lambda)$ be an exact symplectic manifold of dimension $2 m$.
Then the product $M \times W$ is a cooriented contact manifold with contact structure induced by the contact form $\pi_1^* \alpha + \pi_2^* \lambda$, where $\pi_1$ and $\pi_2$ denote the obvious projections.
Indeed, for dimension reasons
\[ d (\pi_1^* \alpha + \pi_2^* \lambda)^{n - 1 + m} = \frac{(n - 1 + m)!}{(n - 1)! \, m!} \, \pi_1^* (d\alpha)^{n - 1} \wedge \pi_2^* (d\lambda)^m \]
and the top-dimensional form
\[ (\pi_1^* \alpha + \pi_2^* \lambda) \wedge d (\pi_1^* \alpha + \pi_2^* \lambda)^{n - 1 + m} = \frac{(n - 1 + m)!}{(n - 1)! \, m!} \, \pi_1^* (\alpha \wedge (d\alpha)^{n - 1}) \wedge \pi_2^* (d\lambda)^m \]
on $M \times W$ is nowhere vanishing.
It is easy to see that the Reeb vector field of the contact form $\pi_1^* \alpha + \pi_2^* \lambda$ is the Reeb vector field of $\alpha$, or more precisely, it is the horizontal (i.e.\ tangent to $M$) lift $\overline{R}_\alpha$ of $R_\alpha$ via the projection map $\pi_1$.
For simplicity, we often write the above contact form as $\alpha + \lambda$ when proper notation becomes too cumbersome unless it is essential for a particular argument.

\subsection{Sign conventions}
There are various conventions in use for the sign of the symplectic form on the symplectization and the sign of the (contact) Hamiltonian function corresponding to a Hamiltonian or contact vector field, and unfortunately, no combination of these choices yields all of the identities one desires.
We choose to define $\omega = d (t \pi_2^* \alpha)$ on the symplectization without a minus sign in order to not introduce an awkward minus sign in the correspondence between the cohomology class represented by a coisotropic embedding and its lift to a Lagrangian embedding into the symplectization.
Then if one wants the lift of a contact isotopy generated by a (contact) Hamiltonian $H$ to be generated by the Hamiltonian $F = t \, \pi_2^* H$ on the symplectization (again without a minus sign), one must introduce a minus sign in the correspondence between a smooth function and either its contact or its Hamiltonian vector field.
Since this paper is mainly on contact geometry, we opt to define $dF = - \omega (X_F, \cdot)$.
This is consistent with formulas in earlier work by the author pertaining to the relationship between a contact Hamiltonian function and its contact vector field and isotopy, and moreover, the Reeb vector field of a contact form $\alpha$ is generated by the constant function $1$ (rather than $- 1$).

\section{Suspension, double, and Gromov's figure eight trick} \label{sec:suspension}

\subsection{The Lagrangian suspension construction}
Let $(W, \omega)$ be a (connected) symplectic manifold without boundary, $\iota \colon L \hookrightarrow W$ be a Lagrangian embedding, and $\{ \varphi_F^s \}$, $s \in S^1 = \R / \Z$, be a loop of Hamiltonian diffeomorphisms with $\varphi_F^0 = \id$, generated by a $1$-periodic Hamiltonian $F \colon \R \times W \to \R$.
Consider the symplectic manifold $W \times T^* S^1$ with the (product) symplectic structure $\pi_1^* \omega + \pi_2^* (dr \wedge ds)$, where $(s, r)$ are coordinates on the factor $T^* S^1 = S^1 \times \R$.
Then the embedding $L \times S^1 \hookrightarrow W \times T^* S^1$ that is defined by $(x, s) \mapsto ( (\varphi_F^s \circ \iota) (x), s, (- F_s \circ \varphi_F^s \circ \iota) (x) )$ is again Lagrangian, called the (Lagrangian) suspension of $\{ \varphi_F^s \circ \iota \}$.
See for instance \cite[Subsection~3.1.E]{polterovich:ggs01} for a (straightforward) proof.

\subsection{The coisotropic suspension construction}
Let $(M, \xi = \ker \alpha)$ and $L$ be as above, $\iota \colon L \hookrightarrow M$ be a coisotropic embedding so that $f \, \iota^* \alpha$ is a closed one-form on $L$, and $\{ \varphi_H^s \}$, $s \in S^1 = \R / \Z$, be a loop of contact diffeomorphisms with $\varphi_H^0 = \id$ and $(\varphi_H^s)^* \alpha = h_s \alpha$, $h_s > 0$, generated by a $1$-periodic contact Hamiltonian function $H \colon \R \times M \to \R$.
Consider the cooriented contact manifold $M \times T^* S^1$ with contact form $\pi_1^* \alpha + \pi_2^* (r \, ds)$, which we frequently write as $\alpha + r \, ds$ to simplify notation.

\begin{lem} \label{lem:suspension}
The embedding $\Phi = \Phi_H \colon L \times S^1 \hookrightarrow M \times T^* S^1$ defined by \[ (x, s) \mapsto \left( (\varphi_H^s \circ \iota) (x), s, \left( - H_s \circ \varphi_H^s \circ \iota \right) (x) \right) \] is coisotropic.
In fact, the one-form $\displaystyle \frac{f}{h_s \circ \iota} \, \Phi^* (\alpha + r \, ds)$ on $L \times S^1$ is closed.
\end{lem}

\begin{proof}
$\Phi^* (\alpha + r \, ds) = (\varphi_H^s \circ \iota)^* \alpha + \alpha ( \frac{d}{ds} (\varphi_H^s \circ \iota) ) ds - (H_s \circ \varphi_H^s \circ \iota) ds = (h_s \circ \iota) \iota^* \alpha$.
\end{proof}

By the lemma, $\Phi$ lifts to a Lagrangian embedding $\widehat{\Phi} \colon L \times S^1 \hookrightarrow \R_+ \times M \times T^* S^1$ into the symplectization of $M \times T^* S^1$ given by
\begin{align} \label{eqn:suspension-lift}
(x, s) \mapsto \left( \frac{f}{h_s \circ \iota} (x), (\varphi_H^s \circ \iota) (x), s, \left( - H_s \circ \varphi_H^s \circ \iota \right) (x) \right).
\end{align}
Up to the obvious identification of $\R_+ \times (M \times T^* S^1)$ with $(\R_+ \times M) \times T^* S^1$ via the exact symplectic diffeomorphism $(t, x, r, s) \mapsto (t, x, t r, s)$, the Lagrangian lift of a coisotropic suspension coincides with the Lagrangian suspension of the Lagrangian lift of the same coisotropic embedding.
One can extend the function $f$ on $L$ to a function $\overline{f}$ on $M$ in the sense that $f = \overline{f} \circ \iota$.
Then the (coisotropic) suspension $\Phi$ in Lemma~\ref{lem:suspension} (and its Lagrangian lift in Equation~\ref{eqn:suspension-lift}) can be considered (for each $s$) as a globally defined map evaluated at $\iota (x)$.

\subsection{A Hofer-type pseudo-norm}
If $H \colon [0, 1] \times M \to \R$ is a compactly supported contact Hamiltonian, define a pseudo-norm (or pseudo-length)
\[ \| H \| = \int_0^1 \left( \sup_{x \in M} H_s (x) - \inf_{x \in M} H_s (x) \right) ds . \]
This quantity is not a norm since it vanishes on functions that depend only on $s$.
Note that if $H$ is a contact Hamiltonian as in the previous subsection used to define the suspension $\Phi$, one can cut off $H$ outside a neighborhood of the compact subset $\bigcup_{s \in S^1} (\varphi_H^s \circ \iota) (L)$ without altering the definition of $\Phi$.

\subsection{The double of a path of diffeomorphisms}
If a given path of contact diffeomorphisms is not a loop, one can construct a loop of contact diffeomorphisms as follows.
Let $\{ \varphi_H^s \}$, $s \in [0, 1]$, be a path of contact diffeomorphisms with $\varphi_H^0 = \id$, generated by a compactly supported contact Hamiltonian $H \colon [0, 1] \times M \to \R$.
Let $\epsilon > 0$.
After reparametrizing the function $H$ if necessary, we may assume that $H_s = 0$ for $s \le \epsilon$ and $s \ge 1 - \epsilon$, without changing the (pseudo-)length of $H$.
Then the contact Hamiltonian function $G$ defined by $G_s = H_{2 s}$ for $s \in [0, 1 / 2]$ and $G_s = - H_{2 (1 - s)}$ for $s \in [1 / 2, 1]$ defines a smooth $1$-periodic contact Hamiltonian on $\R \times M$.
The corresponding loop of contact diffeomorphisms $\{ \varphi_G^s \}$, $s \in S^1 = \R / \Z$, is given by $\varphi_G^s = \varphi_H^{2 s}$ for $s \in [0, 1 / 2]$ and $\varphi_G^s = \varphi_H^{2 (1 - s)}$ for $s \in [1 / 2, 1]$.
Clearly $\| G \| = 2 \| H \|$.
We call the loop $\{ \varphi_G^s \}$ (or its Hamiltonian $G$) the double of the path $\{ \varphi_H^s \}$ (or of the Hamiltonian $H$), and write $(\varphi_G^s)^* \alpha = g_s \alpha$.
The corresponding double construction on a path of Hamiltonian diffeomorphisms is verbatim the same.

\begin{rmk}
Throughout this paper, we make multiple estimates up to $\epsilon > 0$.
In order to simplify notation, we use the same letter for all of these estimates, and often combine them into a new constant, which we continue to denote by $\epsilon$.
\end{rmk}

\subsection{The Gromov figure eight trick}
Let $\iota \colon L \hookrightarrow M$ be a coisotropic embedding so that $f \, \iota^* \alpha$ is a closed one-form on $L$, and $\{ \varphi_H^s \}$, $s \in [0, 1]$, be a path of contact diffeomorphisms that is generated by a compactly supported contact Hamiltonian $H \colon [0, 1] \times M \to \R$ such that $\varphi_H^0 = \id$ and $(\varphi_H^1 \circ \iota) (L) \cap \iota (L) = \emptyset$.
Let $G$ denote the contact Hamiltonian generating the double $\{ \varphi_G^s \}$, $s \in S^1 = \R / \Z$, of the path $\{ \varphi_H^s \}$ as defined in the previous subsection, and consider the (coisotropic) suspension $\Phi_G (x, s) = \left( (\varphi_G^s \circ \iota) (x), s, \left( - G_s \circ \varphi_G^s \circ \iota \right) (x) \right)$ of $\{ \varphi_G^s \circ \iota \}$.
Let $\epsilon > 0$, and define $a_+ (s) = - \inf_{x \in M} (G_s \circ \varphi_G^s) (x) + \epsilon$ and $a_- (s) = - \sup_{x \in M} (G_s \circ \varphi_G^s) (x) - \epsilon$.
Denote by $C \subset T^* S^1 = S^1 \times \R$ the annulus $\{ (s, r) \in S^1 \times \R \mid a_- (s) < r < a_+ (s) \}$.
Then $\Phi_G (L \times S^1) \subset M \times C$.

Now consider Gromov's figure eight trick immersion $\psi \colon C \to \R^2$, see for example \cite[Section~3.3, Step~3]{polterovich:ggs01}.
This symplectic immersion takes the zero section $\{ r = 0 \}$ of $T^* S^1 = S^1 \times \R$ to a figure eight curve with ears of equal area, so that the (closed) one-form $\psi^* (p dq) - r ds = d\eta$ is exact, where $(p, q)$ are coordinates on $\R^2$, and where $d (p dq) = dp \wedge dq$ is the standard symplectic form on $\R^2$.
If $\mu$ is a smooth function on $\R^2$, then $\psi^* (p dq - d\mu) - r ds = d(\eta - \mu \circ \psi)$.
We can choose the immersion $\psi$ and a compactly supported smooth function $\mu$ so that $(\eta - \mu \circ \psi) (r, s) = - \chi (s) r s$, where $\chi \colon \R \to [0, 1]$ is a smooth function that vanishes outside of the interval $- \epsilon < s < \epsilon$.
We replace the primitive $p dq$ with the primitive $p dq - d\mu$ of $d p \wedge dq$.
Since $G_s = 0$ for $- \epsilon \le s \le \epsilon$, the composition $(\id \times \psi) \circ \Phi_G \colon L \times S^1 \to M \times \R^2$ is a coisotropic immersion.
In fact,
\[ \frac{f}{g_s \circ \iota} \left( ((\id \times \psi) \circ \Phi_G)^* (\alpha + p dq - d\mu) \right) = f \, \iota^* \alpha. \]
Strictly speaking, the latter is the closed one-form $\pi_1^* (f \, \iota^* \alpha)$ on $L \times S^1$.
Moreover, since $(\varphi_H^1 \circ \iota) (L) \cap \iota (L) = \emptyset$, the immersion $(\id \times \psi) \circ \Phi_G$ is in fact an embedding.
The (easy) proof is verbatim the same as in the Lagrangian case in \cite[Section~3.3, Step~5]{polterovich:ggs01} (possible double-points could have occurred only for pairs of points with $- \epsilon < s < \epsilon$ and $- \epsilon + 1 / 2 < s < \epsilon + 1 / 2$, respectively).

If $\| H \| = l$, then $\| G \| = 2 l$, and one can carry out the above construction so that the image $\psi (C) \subset \R^2$ is contained in a disk $B$ of area less than $2 l + 3 \epsilon$.

\begin{rmk}
In this paper, we frequently consider the projection of a product manifold to one of its factors.
In order to simplify notation, we denote all of them by the letter $\pi$ with a subscript to identify the factor to which we project.
\end{rmk}

\subsection{Deforming the product contact form}
In light of Equation~\ref{eqn:lambda} and the relation $\rho^* \lambda_0 = t \pi_2^* \alpha_0$, we wish to replace the one-form $p dq - d\mu$ considered in the previous subsection by the one-form $\frac{1}{2} (p dq - q dp) = p dq - \frac{1}{2} d (p q)$.

\begin{lem}
Let $(M, \xi = \ker \alpha)$ be a cooriented contact manifold and $(W, \omega = d\lambda)$ be an exact symplectic manifold.
Let $V \subset M$ be an open subset with compact closure, and $\lambda_t$ be a compactly supported family of closed one-forms on $W$ with $\lambda_0 = 0$.
Then there exists a compactly supported isotopy $\varphi_t$ of diffeomorphisms of $M \times W$ such that $\varphi_t^* (\pi_1^* \alpha + \pi_2^* (\lambda + \lambda_t)) = \pi_1^* \alpha + \pi_2^* \lambda$ on $V \times W$.
\end{lem}

\begin{proof}
We use a standard Gray stability argument.
Suppose that $\varphi_t$ is an isotopy of diffeomorphisms with $\varphi_0 = \id$ that is generated by a (time-dependent) vector field $X_t$, and 
differentiate the desired identity $\varphi_t^* (\pi_1^* \alpha + \pi_2^* (\lambda + \lambda_t)) = f_t (\pi_1^* \alpha + \pi_2^* \lambda)$ with respect to $t$, where $f_t$ is a family of smooth functions on $M \times W$ yet to be determined.
Since $\frac{d}{dt} (\pi_1^* \alpha + \pi_2^* (\lambda + \lambda_t)) = \pi_2^* (\frac{d}{dt} \lambda_t)$ and $d(\pi_1^* \alpha + \pi_2^* (\lambda + \lambda_t)) = (\pi_1^* d\alpha + \pi_2^* d\lambda)$, it suffices to solve the equation $\pi_2^* (\frac{d}{dt} \lambda_t) = - (\pi_1^* d\alpha + \pi_2^* d\lambda) (X_t)$ for a family of vector fields $X_t$ that are tangent to $\ker (\pi_1^* \alpha + \pi_2^* (\lambda + \lambda_t))$.
Note that in this case $f_t$ is independent of $t$, and thus $f_t = f_0 = 1$.

The vector fields $X_t$ can be constructed in three steps.
Since the two-form $\omega = d\lambda$ is non-degenerate, there exists a smooth family $Z_t$ of vector fields on $W$ so that $\pi_2^* (\frac{d}{dt} \lambda_t) = - (\pi_2^* d\lambda) (\overline{Z}_t)$, where $\overline{Z}_t$ denotes the horizontal lift of $Z_t$ (i.e. the lift that is tangent to $W$).
Then define a family of smooth vector fields $Y_t$ that are tangent to $M$ by $(\pi_1^* d\alpha) (Y_t) = 0$ and $(\pi_1^* \alpha) (Y_t) = - (\pi_2^* (\lambda + \lambda_t)) (\overline{Z}_t)$, or in other words, $Y_t = - ((\pi_2^* (\lambda + \lambda_t)) (\overline{Z}_t)) \overline{R}_\alpha$.
The only thing left to show is that the family of vector fields $X_t = Y_t + \overline{Z}_t$ can be uniquely integrated for all $t$.
This can be achieved by replacing $\overline{Z}_t$ by $(\pi_1^* \kappa) \overline{Z}_t$ for a cut-off function $\kappa$ on $M$ with $\kappa = 1$ on $V$ and compact support in $M$.
\end{proof}

We invoke the lemma with $W = B$, $\lambda = \frac{1}{2} (p dq - q dp)$, and $\lambda_t = t d(\upsilon \frac{1}{2} p q - \mu)$, where $\upsilon$ is a compactly supported cut-off function on the disk $B$ with $\upsilon = 1$ on the (projection to $B$ of the) image of the suspension $(\id \times \psi) \circ \Phi_G$, and $V$ is a neighborhood of the image of the (projection to $M$ of the) suspension.
Note that even in this special case the equation in the lemma can not necessarily be solved for an isotopy that is trivial on the first factor and a Hamiltonian isotopy on the second factor.
Let $\Psi = \varphi_1^{- 1} \circ (\id \times \psi) \circ \Phi_G$, where $\varphi_1$ is as in the lemma.
Then
\[ \frac{f}{g_s \circ \iota} \left( \Psi^* \left( \alpha + \frac{1}{2} (p dq - q dp) \right) \right) = \pi_1^* (f \, \iota^* \alpha). \]
In particular, if $U$ is an open subset of $\R^{2 n - 1}$, $\overline{V} \subset U$, and the symplectic embedding $\rho$ and one-form $\lambda_0$ on $\R^{2 n + 2}$ are defined as in Equations~\ref{eqn:rho} and \ref{eqn:lambda}, then
\begin{align} \label{eqn:cohomology-class}
(\rho \circ \widehat{\Psi})^* \lambda_0 = \frac{f}{g_s \circ \iota} \left( \Psi^* \left( \alpha_0 + \frac{1}{2} (p dq - q dp) \right) \right) = \pi_1^* (f \, \iota^* \alpha_0).
\end{align}
Here $\widehat{\Psi} \colon L \times S^1 \to C (U \times B) = \R_+ \times (U \times B)$ denotes the Lagrangian lift of the coisotropic suspension $\Psi \colon L \times S^1 \to U \times B$.

\section{Construction of non-constant holomorphic disk} \label{sec:holo-disk}

\subsection{Gromov compactness and persistence principle} \label{subsec:gromov-compact}
In this subsection we review Gromov's compactness theorem and the corresponding persistence principle \cite[Section~1.5]{gromov:pcs85}.
For the sake of simplicity, we only treat the theory in the generality necessary for the applications in this paper.
A good exposition (which we mostly follow below but which does not treat all aspects of the theory either) can be found in \cite[Chapter~4]{polterovich:ggs01}.
The standard reference for the general theory is the book \cite{mcduff:hcs04}.

Denote by $D$ the unit disk in the complex plane $\C$.
Let $L \subset \C^n$ be a closed Lagrangian submanifold, and $v \colon D \times \C^n \to \C^n$ be a smooth map that is bounded together with all of its derivatives.
Fix a relative homology class $\beta \in H_1 (\C^n, L)$, and consider the problem $P (L, v, \beta)$ of finding a smooth map $u \colon (D, \p D) \to (\C^n, L)$ such that $\pbar u (\zeta) = v (\zeta, u (\zeta))$ and $[ u ] = \beta$.
Assume that $\{ v_k \}$ is a sequence of smooth functions that $C^\infty$-converges to a smooth function $v$, and that $u_k$ are solutions of the corresponding problems $P (L, v_k, \beta)$.
Gromov's compactness theorem states that there exists a subsequence of $\{ u_k \}$ that converges either to a solution of $P (L, v, \beta)$ or Gromov converges (or converges weakly) to a so-called cusp solution.
In particular, if the problem $P (L, v, \beta)$ has no solution, then bubbling off takes place.
That means that there exists a decomposition $\beta = \beta' + \beta_1 + \cdots + \beta_m$, $\beta_j \not=0$, $m \ge 1$, a solution $u$ of $P (L, v, \beta')$, and solutions $w_j$ of $P (L, 0, \beta_j)$.
The latter are therefore (non-constant) holomorphic disks $(D, \p D) \to (\C^n, L)$.

For later reference, we take a closer look at the notions of cusp curve and Gromov convergence.
Consider a finite system of disjoint simple curves $\gamma_\nu$ that are either closed and lie in the interior of $D$ or have endpoints (and their endpoints only) on the boundary $\p D$.
Denote by $D^0$ the surface (with boundary) obtained from $D \, \backslash \bigcup \gamma_\nu$ by the one-point compactification at every end of $D \, \backslash \bigcup \gamma_\nu$, and by $\overline{D}$ the space obtained from $D$ by shrinking every curve $\gamma_\nu$ to a single point.
Topologically, $\overline{D}$ is a disk with spheres attached at points in its interior, and disks attached at points on its boundary.
There is an obvious map $D^0 \to \overline{D}$, and the Cauchy-Riemann operator $\pbar$ can be defined on a cusp curve $\overline{u} \colon \overline{D} \to \C^n$ by means of the composition $D^0 \to \overline{D} \to \C^n$.
Gromov convergence now means that the following conditions are met.
The maps $u_k$ converge to $\overline{u}$ uniformly on compact subsets of $D \, \backslash \bigcup \gamma_\nu$, and uniformly (on $D$) to the composed map $D \to \overline{D} \to \C^n$, where the first map is the obvious quotient map, and the second map is $\overline{u}$.
Note that for topological reasons there exist no non-constant holomorphic spheres in $\C^n$ and hence in the above decomposition.
Here the maps $u$ and $w_j$ considered above are defined on the (one-point compactifications of the) components of $D \, \backslash \bigcup \gamma_\nu$ (with the latter being defined on the bubbles).
Finally, the areas of the maps $u_k$ converge to the area of the map $\overline{u}$, which is equal to the sum $\area (u) + \sum_j \area (w_j)$.
Again for later reference, recall that the area of a smooth map $w \colon D \to \C^n$ satisfies the inequalities \[ \omega ([w]) = \int_D w^* \omega \le \area (w) \le 2 \int_D | \pbar w |^2 dx dy + \omega ([w]). \]
In particular, the area of a holomorphic map coincides with its symplectic area.

As a consequence of his compactness theorem, Gromov proved a persistence principle that stipulates the following.
Given a generic family $v_s$, $s \in [0, 1]$, with $v_0 = 0$, either the problem $P (L, v_s, 0)$ has a solution for all $s$, or bubbling off occurs at some $s_\infty \le 1$, that is, there exists a subsequence $s_k \to s_\infty$ such that a sequence of solutions of $P (L, v_{s_k}, 0)$ converges to a cusp solution of $P (L, v_{s_\infty}, 0)$.

\subsection{A priori bounds on the size of the image of a holomorphic curve}
Gromov compactness depends on the fact that holomorphic curves into $\C^n$ (and other non-compact symplectic manifolds that are convex at infinity) cannot leave an appropriate compact subset, and a quantitative version of this is required below.
Let $B^{2 n} = \{ z \in \C^n \mid | z - \tau | < T \}$ denote the ball of radius $T > 0$ centered at $\tau \in \C^n$, and let $w \colon D \to \C^n$ be a holomorphic disk so that $w (\p D) \subset B$.
Then a standard convexity argument shows that the image $w (D)$ is also contained in $B$, cf.\ the proof of Proposition~9.2.16 in \cite{mcduff:hcs04}.
In particular, the image of $w$ is contained in the strip $\{ z \in \C^n \mid \tau_1 - T < z_1 < \tau_1 + T \}$, where the subscript $1$ denotes the first real coordinate of a complex vector.

\subsection{Disks with boundary on coisotropic submanifolds} \label{subsec:disks}
Let $U$ be an open subset of $\R^{2 n - 1}$, and $\iota \colon L \hookrightarrow U$ be a coisotropic embedding.
Denote by $SL$ the image of the Lagrangian lift $\rho \circ \widehat{\Psi} \colon L \times S^1 \to C (U \times B) \to \R^{2 n + 2}$, where $B \subset \R^2$ is a disk of radius $R$ with $\pi R^2 < 2 l + 3 \epsilon$.
Identify $\R^{2 n + 2}$ with $\C^{n + 1} = \C^n \times \C$, and consider as in subsection~\ref{subsec:gromov-compact} the problem $P (SL, v, 0)$ of finding a smooth map $u \colon (D, \p D) \to (\C^{n + 1}, SL)$ such that $\pbar u (\zeta) = v (\zeta, u (\zeta))$ and $[ u ] = 0$.
In this case, the appropriate function $v$ to consider is $v (\zeta, \rho (t, z, x, y, p, q)) = (0, \ldots, 0, \sqrt{t} \, \sigma)$ on the cone $C (\R^{2 n + 1})$, where $\sigma \in \C$ is a constant.
Of course the latter and its derivative are not bounded, so that one does not have the a priori area and first derivative bounds that are necessary for Gromov compactness.
We rectify this situation by replacing $\sqrt{t}$ by a smooth function $\Lambda \colon \C^{n + 1} \to \R$ that is bounded together with all of its derivatives, and satisfies $\Lambda (\rho (t, z, x, y, p, q)) = \sqrt{t}$ on the subset $\rho ([T_1, T_2] \times \R^{2 n + 1}) = \{ z \in \R^{2 n + 2} \mid T_1 < z_1 < T_2 \}$, where $0 < T_1 < T_2$ are constants to be specified momentarily.

Let $u$ be a solution of $P (SL, v, 0)$.
Observe that the image $u (\partial D)$ is contained in the subset
\[ \left\{ z \in \C^{n + 1} \left| \min_{(x,s) \in L \times S^1} \frac{f (x)}{(g_s \circ \iota) (x)} < z_1 < \max_{(x,s) \in L \times S^1} \frac{f (x)}{(g_s \circ \iota) (x)} \right. \right\}, \]
and that the projection $\pi_1 \circ u \colon D \to \C^n$ is holomorphic.
Apply the a priori bound in the previous subsection with a sufficiently large ball that is contained in the half-space $\{ z \in \C^n \mid z_1 > 0 \}$ and in turn contains (the projection to $\C^n$) of the Lagrangian submanifold $SL$.
That implies the existence of constants $0 < T_1 < T_2$ so that the image of every solution $u$ of $P (SL, v, 0)$ is contained in the subset $\{ z \in \C^{n + 1} \mid T_1 < z_1 < T_2 \}$.

\begin{lem}
Let $\sigma \in \C$ so that $| \sigma | > R$.
Then the problem $P (SL, v, 0)$ with $v$ as above has no solutions.
\end{lem}

\begin{proof}
Suppose that $u$ is a solution of $P (SL, v, 0)$, and denote by $\phi$ its last complex coordinate.
Recall the projection map $\Pi \colon C(\R^{2 n + 1}) \to \R^{2 n + 2} = \C^{n + 1}$ defined in subsection~\ref{subsec:symplectization}.
Since $\sigma$ is a planar vector, we have $\p (\Pi \circ \phi) / \p x + i (\p (\Pi \circ \phi) / \p y) = 2 \Pi_* \pbar \phi = 2 \sigma$.
Moreover, $| (\Pi \circ \phi) |_{\p D} | \le R$ since $u (\p D) \subset SL$.
By the same argument as in the proof of \cite[Lemma~4.3.A]{polterovich:ggs01}, we obtain the identity
\[ 2 \pi \sigma = \int_D 2 \pbar (\Pi \circ \phi) dx dy = \int_D d ((\Pi \circ \phi) dy - i (\Pi \circ \phi) dx) = \int_{\p D} (\Pi \circ \phi) dy - i (\Pi \circ \phi) dx \]
and the upper bound
\[ \left| \int_{\p D} (\Pi \circ \phi) dy - i (\Pi \circ \phi) dx \right| = 2 \pi \left| \int_0^1 e^{2 \pi i \theta} (\Pi \circ \phi) (e^{2 \pi i \theta}) d\theta \right| \le 2 \pi R, \]
and thus $| \sigma | < R$.
\end{proof}

Apply Gromov's persistence principle to the family $s v$, $s \in [0, 1]$, where $v$ is as in the lemma.
For the sake of simplicity, we assume that bubbling off happens for the family $s v$ (rather than a $C^\infty$-small perturbation); all of the estimates below hold up to $\epsilon > 0$, and therefore the argument goes through without significant changes.
We then obtain a sequence $s_k \to s_\infty \le 1$, a decomposition $0 = \beta' + \beta_1 + \cdots + \beta_m$, $\beta_j \not= 0$, solutions $u_k$ of $P (SL, s_k v, 0)$, $u$ of $P (SL, s_\infty v, \beta')$, and holomorphic disks $w_1, \ldots, w_m$, $m \ge 1$, with $[ w_j ] = \beta_j$, such that the sequence $u_k$ Gromov converges to a cusp solution $\overline{u}$ with the above data.

We deduce the following estimates.
The area of $\Pi \circ u_k$ is bounded from below by
\[ \area (\Pi \circ u_k) \le 2 \int_D | \pbar (\Pi \circ u_k) |^2 dx dy = 2 \pi s_k^2 | \sigma |^2 \le 2 \pi | \sigma |^2. \]
On the other hand, the area of $(\Pi \circ w_1)$ is less than the area of $(\Pi \circ u_k) + \epsilon$ provided that $k$ is sufficiently large.
This follows from the convergence of the area and the uniform convergence on compact subsets in the Gromov convergence of the sequence $u_k$ to the cusp curve $\overline{u}$.
Finally, since $[ w_1 ] = \beta_1 \not= 0$, the curve $w_1$ is non-constant, and in particular has positive area.
Since its boundary lies on $SL$, the definition of the map $\Pi$ and Equation~\ref{eqn:suspension-lift} for the lift of a coisotropic suspension imply that the area of the disk $w = \Pi \circ w_1$ is also positive.
Note that the above argument holds for any $\sigma \in \C$ with $| \sigma | > R$.

\begin{proof}[Proof of Theorem~\ref{thm:non-constant-disk}]
The map $w \colon (D, \p D) \to (\R^{2 n + 1}, \Psi (L \times S^1))$ constructed in the preceding paragraphs is precisely the holomorphic disk in the statement of the theorem.
Its area $a$ satisfies $0 < a < 4 \| H \| + \epsilon$.

The map $\p D \to L$ in the last part of the statement is given by the composition $\pi_1 \circ \Psi^{-1} \circ (w |_{\p D})$, where $\pi_1 \colon L \times S^1 \to L$ as before is the projection.
This map is unchanged if we replace $w$ by $w_1$ and $\Psi$ by its Lagrangian lift $\rho \circ \widehat{\Psi}$ into $\R^{2 n + 2}$.
Then by Equation~\ref{eqn:cohomology-class}
\[ \int_{(\pi_1 \circ (\rho \circ \widehat{\Psi})^{-1} \circ w_1) (\p D)} f \, \iota^* \alpha = \int_{w_1 (\p D)} \lambda_0 = \int_{w_1 (D)} \omega_0 = \area (w_1) \not= 0, \]
and thus the map $\p D \to L$ represents a non-trivial element of $H_1 (L, \R)$.
\end{proof}

\begin{proof}[Proof of Theorem~\ref{thm:empty-shape}]
Let $H \colon [0, 1] \times \R^{2 n - 1} \to \R$ be a Hamiltonian function whose contact vector field $X_H$ is nowhere tangent to $\jmath (L)$.
If we cut $H$ off outside an arbitrary neighborhood of $\jmath (L)$ and then multiply $H$ by a (small) positive constant, the resulting contact vector field is still nowhere tangent to $\jmath (L)$.
Thus we may assume that $H$ is compactly supported in a tubular neighborhood $U$ of $\jmath (L)$ so that the closure $\overline{U} = K$ is compact, and that the inequality $4 \| H \| < b (K)$ holds, where $b (K)$ is as in Lemma~\ref{topo-area-bound}.
Moreover, since $(\varphi_H^1 \circ \jmath) (L) \cap \jmath (L) = \emptyset$ and $L$ is compact, there exists a tubular neighborhood $N$ of $\jmath (L)$ that is displaced by $\varphi_H^1$ as well.

In order to derive a contradiction, suppose $\iota \colon L \to N$ is a coisotropic embedding so that the induced homomorphism $\iota_* \colon H_1 (L, \R) \to H_1 (N, \R)$ is injective (i.e.\ an isomorphism).
By Theorem~\ref{thm:non-constant-disk}, there exists $\epsilon > 0$ and a non-constant holomorphic disk $w$ of area $a < 4 \| H \| + \epsilon < b (K)$ whose boundary is non-trivial in $H_1 (L, \R)$.
Since the area of $w$ is larger than the area of $\pi_1 \circ w$, we must have $a \ge b (K)$, which is the desired contradiction.
\end{proof}

\section{Rational classes and energy-capacity type inequalities} \label{sec:limitations}

Let $\iota \colon L \hookrightarrow U \subset \R^{2 n - 1}$ be a coisotropic embedding as in Theorem~\ref{thm:non-constant-disk}, and assume that the cohomology class $z = [f \, \iota^* \alpha_0] \in H^1 (L, \R)$ is rational.
Recall that this means that the subgroup $z (H_1 (L, \Z)) \subset \R$ is discrete.
Note that this definition is invariant under rescaling by a non-zero constant.
Denote by $\gamma > 0$ the positive generator of the group $z (H_1 (L, \Z))$.
Then by Equation~\ref{eqn:cohomology-class}, the area of $w_1$ belongs to the image $z (H_1 (L, \Z))$, and thus $\area (w_1) \ge \gamma > 0$.
In the Lagrangian analog of Theorem~\ref{thm:non-constant-disk}, one can therefore bound the quantity $4 \| F \|$ from below by the generator $\gamma$, where $F$ is a Hamiltonian that displaces the rational Lagrangian embedding.
(In fact, by the decomposition $0 = \beta' + \beta_1 + \cdots + \beta_m$, $\beta_j \not= 0$ in the Gromov convergence in subsection~\ref{subsec:disks}, one has either $\beta' \not= 0$ or $m > 1$, so that one can improve the factor $4$ to a factor $2$.)
This gives rise to an energy-capacity type inequality bounding the norm $\| F \|$ from below by the symplectic size of a subset that is displaced by the time-one map $\varphi_F^1$, see \cite[Subsection~3.2.D]{polterovich:ggs01}.
In this section we briefly discuss the limitations of this argument for coisotropic embeddings.

The problem arises from the fact that the area of the disk $w_1$ does not necessarily belong to the subgroup $z (H_1 (L, \Z))$, where $z = [f \, \iota^* \alpha_0] \in H^1 (L, \R)$ is as above.
(It is not difficult to see that $f \not= 1$, see e.g.\ \cite{mueller:csc16}.)
One can however bound the area of $w_1$ from above by $C (f) \, \area (w)$, where $C (f)$ is a constant that depends on the $C^1$-norm of $f$, and thus derive the inequality $0 < \gamma / C (f) \le 4 \| H \| + \epsilon$.
One can choose this constant so that the ratio $\gamma / C (f)$ is invariant under rescaling $f$ by a positive constant.
One can improve the constant $C (f)$ further by considering the Lagrangian lift of $\iota$ into the symplectization and a suitable cut-off of the Hamiltonian $F = \pi_2^* H$.
In that case the constant $C (f)$ depends only on $\max (f)$.

This estimate however has limited practical applications, since there is no control over (the magnitude of) the function $f$ so that the one-form $f \, \iota^* \alpha_0$ is closed.
Indeed, for any $x_0 \in L$ and $\epsilon > 0$ an arbitrary constant, one easily constructs a compactly supported contact diffeomorphism $\psi \colon U \to U$ so that $\psi^* \alpha_0 = h \alpha_0$ and $h (\iota (x_0)) = \epsilon$.
Then the embedding $\psi \circ \iota \colon L \hookrightarrow U$ is coisotropic, and $(f / h \circ \iota) (\psi \circ \iota)^* \alpha_0 = f \, \iota^* \alpha_0 = z$.
If $f$ assumes its maximum at $x_0 \in L$, then $\max (f / h \circ \iota) \ge \max (f) / \epsilon$.
Thus the above estimate does not provide an obstruction to coisotropic embeddings in terms of the positive generator $\gamma$ of a rational cohomology class $z \in H^1 (L, \R)$.

\section{Proof of Theorem~\ref{thm:non-exact-coiso}} \label{sec:proof}

The first result of this section is a straightforward generalization of Gromov's theorem that there exists no exact Lagrangian embedding into standard symplectic $(\R^{2 n}, \omega_0 = d\lambda_0)$ \cite[Corollary~2.3.B$_{\textrm 2}$]{gromov:pcs85}.

\begin{thm}
Let $(X, \omega = d\lambda)$ be an exact symplectic manifold without boundary of dimension $2 n - 2$, and let $W = X \times \R^2$ with the product symplectic structure $\pi_1^* \omega + \pi_2^* \omega_0$.
Let $\iota \colon L \hookrightarrow W$ be a Lagrangian embedding.
Then the cohomology class $[\iota^* (\pi_1^* \lambda + \pi_2^* \lambda_0)]$ does not belong to the image of the induced homomorphism $\iota^* \colon H^1 (W, \R) \to H^1 (L, \R)$.
In other words, if $Q \colon H^1 (L, \R) \to H^1 (L, \R) / \textrm{Im} (\iota^*)$ denotes the canonical quotient map, then $Q ([\iota^* (\pi_1^* \lambda + \pi_2^* \lambda_0)]) \not= 0$.
\end{thm}

\begin{proof}
Gromov compactness \cite[Subsection~2.3.B$_{\textrm 3}$]{gromov:pcs85} implies the existence of either a non-constant $J$-holomorphic disk $w \colon (D, \partial D) \to (W, \iota (L))$ or a non-constant $J_X$-holomorphic sphere $S^2 \to X$, where $J = J_X \oplus J_0$ is an auxiliary almost complex structure on $W = X \times \R^2$ that is standard on the second factor.
The latter possibility however is excluded by exactness of $\omega$.
Let $S^1$ denote a cycle in $L$ such that $\iota (S^1) = w (\partial D)$.
Then $\int_{S^1} \iota^* (\lambda \oplus \lambda_0) = \area (w) \not= 0$.
On the other hand, if $\theta$ is a closed one-form on $W$, then $\int_{S^1} \iota^* \theta = \int_{\iota (S^1)} \theta = \int_{w (D)} d\theta = 0$.
\end{proof}

As a special case, we state the following theorem.

\begin{thm} \label{thm:non-exact-lag}
Consider $W = T^*S^1 \times \R^{2 n}$ with its standard symplectic structure $d(\pi_1^* (r \, ds) + \pi_2^* \lambda_0)$, and suppose that $\iota \colon S^1 \times L \hookrightarrow W$ is a Lagrangian embedding so that $\iota^* = \iota_0^* \colon H^1 (W, \R) \to H^1 (S^1, \R) \oplus H^1 (L, \R)$, where $\iota_0$ denotes the canonical map $S^1 \times L \to S^1 \times 0 \to T^*S^1 \times \R^{2 n}$.
Consider the corresponding cohomology class $[\iota^* (\pi_1^* (r \, ds) + \pi_2^* \lambda_0)] = (A, Z) \in H^1 (S^1, \R) \oplus H^1 (L, \R)$.
Then $Z \not= 0$.
\end{thm}

\begin{proof}[Proof of Theorem~\ref{thm:non-exact-coiso}]
Recall that $\alpha_{\textrm{can}} = dz + \frac{1}{2} \sum_{j = 1}^n (x_j dy_j - y_j dx_j)$.
We can identify the symplectization $\R_+ \times (S^1 \times \R^{2 n})$ with a subset of $T^* S^1 \times \R^{2 n}$ via a symplectic embedding $\rho$ that is defined verbatim as in Equation~\ref{eqn:rho}.
The statement then follows directly from Theorem~\ref{thm:non-exact-lag} applied to the Lagrangian embedding $\rho \circ \hat{\iota}$.
Here we use the fact from subsection~\ref{subsec:coisotropic} that a coisotropic embedding and its Lagrangian lift represent the same cohomology class in $H^1 (S^1 \times L, \R)$.
\end{proof}

\bibliography{holo-disks}
\bibliographystyle{plain}

\end{document}